\documentclass[12pt, reqno]{amsart}

\usepackage{amssymb}
\usepackage{amsmath}
\usepackage{amsthm}
\usepackage[mathscr]{eucal}
\usepackage{enumerate}

\usepackage{geometry, graphicx, color} 

\usepackage{verbatim} 

\usepackage{tikz} 

\theoremstyle{plain}
\newtheorem{theorem}{Theorem}[section]

\theoremstyle{plain}
\newtheorem{prop}{Proposition}[section]

\theoremstyle{definition}

\theoremstyle{remark}
\newtheorem{remark}[theorem]{Remark}

\theoremstyle{plain}

\theoremstyle{plain}
\newtheorem{lemma}[theorem]{Lemma}

\newcommand{\del}{\partial}

\def\R{{\mathbb{R}}}
\def\S{{\mathbb{S}}}
\def\N{{\mathbb{N}}}

\def\FT{{\mathcal{F}}}

\def\sA{{\mathcal{A}}}

\def\sS{{\mathcal{S}}}

\def\P{{\mathbf{P}}}

\newcommand{\G}{\mbox{$\Gamma$}}

\newcommand{\<}{\langle}
\renewcommand{\>}{\rangle}

\renewcommand{\^}[1]{\widehat{#1}}

\renewcommand{\~}[1]{\widetilde{#1}}

\usepackage{setspace}

\begin{document}

\title[Improved LWP for quadratic derivative NLW in 2D]{Improved well-posedness for the quadratic derivative nonlinear wave equation in 2D}


\author[Grigoryan]{Viktor Grigoryan$^1$}
\address{$^1$  Department of Mathematics and Statistics \\ Simmons College  \\  300 The Fenway, Boston, Massachusetts, 02115}
\email{grigoryan@simmons.edu}
\thanks{}

\author[Tanguay]{Allison Tanguay$^2$}
\address{$^2$  
Department of Mathematics and Statistics\\ University of Massachusetts Amherst\\ Amherst, Massachusetts, 01003}
\email{ tanguay@math.umass.edu}
\thanks{}

\begin{abstract}
In this paper we consider the Cauchy problem for the nonlinear wave equation (NLW) with quadratic derivative nonlinearities in two space dimensions. Following Gr\"{u}nrock's result in 3D, we take the data in the Fourier-Lebesgue spaces  $\^{H}_s^r$, which coincide with the Sobolev spaces of the same regularity for $r=2$, but scale like lower regularity Sobolev spaces for $1<r<2$. We show local well-posedness (LWP) for the range of exponents $s>1+\frac{3}{2r}$, $1<r\leq 2$. On one end this recovers the sharp result on the Sobolev scale, $H^{\frac{7}{4}+}$, while on the other end establishes the $\^{H}_{\frac{5}{2}}^{1+}$ result, which scales like the Sobolev $H^{\frac{3}{2}+}$, thus, corresponding to a $\frac{1}{4}$ derivative improvement on the Sobolev scale.
\end{abstract}

\maketitle




\section{Introduction}
Consider the quadratic derivative nonlinear wave equation (NLW),
\begin{equation}\label{main_eq}
 \Box u=(\del u)^2 \qquad \text{in } \R^{1+n},
\end{equation}
where $\del u$ is the space-time gradient of $u$, and no special structure is assumed in the nonlinearity. We study the local well-posedness (LWP) question for equation \eqref{main_eq} with initial data taken in Fourier-Lebesgue spaces $\^{H}_s^r$. Thus, we consider the Cauchy problem for \eqref{main_eq} with initial conditions
\begin{equation}\label{main_IC}
(u, \del_t u)|_{t=0}=(f, g)\in \^{H}_s^r\times \^{H}_{s-1}^r.
\end{equation}
Our goal is to establish local well-posedness for a range of the exponents $(r, s)$, which improves on previously known Sobolev space results. The Fourier-Lebesgue spaces $\^{H}_s^r$ have previously appeared in literature in the context of various equations, and were used to achieve improved regularity results. See for example \cite{hormander}, \cite{CVV}, \cite{vargas-vega}, \cite{grunrock:mkdv}, \cite{grunrock:nls}, \cite{christ}, \cite{grunrock-herr}, \cite{grunrock:wave}.

Equation \eqref{main_eq} is invariant under the scaling
\begin{equation}\label{scaling}
 (t, x)\mapsto (\lambda t, \lambda x).
\end{equation}
That is, if $u$ is a solution to \eqref{main_eq} in $\R^{1+n}$, then so is the scaled function $u_\lambda (t, x)=u(\lambda t, \lambda x)$. The homogeneous Sobolev norm of the initial data then scales as
\[
 \|u_\lambda(0, \cdot)\|_{\dot{H}^s(\R^n)}=\lambda^{s-\frac{n}{2}}\|u(0, \cdot)\|_{\dot{H}^s(\R^n)},
\]
and $s_c=\frac{n}{2}$ is called the critical exponent for equation \eqref{main_eq}, as the $\dot{H}^{s_c}$ norm of the initial data is preserved under its scaling. From purely scaling considerations, one might expect that local well-posedness holds for data in the Sobolev space $H^s$ for $s>s_c$ (subcritical regime), global existence holds for small data in $\dot{H}^{s_c}$ (critical regime), and some form of ill-posedness holds for data in $H^s$ for $s<s_c$ (supercritical regime). However, counterexamples of Lindblad \cite{lindblad:1}, \cite{lindblad:2}, \cite{lindblad:3} further restrict the regularity assumption on initial data for local well-posedness. In view of these counterexamples, one can only expect local well-posedness for \eqref{main_eq} for
\begin{equation}\label{s_range}
 s>\left\{\frac{n}{2}, \frac{n+5}{4}\right\}.
\end{equation}
For dimensions $n\geq 5$ this range coincides with the subcritical range $s>\frac{n}{2}$, for which local well-posedness was proved by Tataru \cite{tataru-dqnlw} using refinements of $X^{s, b}$ spaces.

In dimensions $n=4$ one can show local well-posedness for the range $s>\frac{9}{4}$, which is sharp by \eqref{s_range}, using the regular $X^{s, b}$ spaces (see \cite{foschi-klainerman:bilinear} for related bilinear estimates).

In dimensions $n=3$, Ponce and Sideris  \cite{ponce-sideris} proved LWP for $s>2$ using Strichartz estimates approach. This is again sharp in light of the counterexamples of Linblad. If the nonlinearity has a null-form structure, Klainerman and Machedon \cite{kl-mac:null}  improved the range for LWP to $s>\frac{3}{2}$, which is almost critical in 3D.

In dimension $n=2$, the sharp LWP result for $s>\frac{7}{4}$ can be again shown using the Strichartz estimates approach. As we were unable to locate a reference for this result, we sketch its easy proof in Appendix \ref{appA}. If the nonlinearity has a null-form structure, the LWP results can again be improved. For the $Q_0$ null-form, almost critical LWP was established by Klainerman and Selberg \cite{kl-sel} in the context of wave maps. For the other null forms, Zhou \cite{zhou} established LWP for data in $H^s$ with $s>\frac{5}{4}$, and also showed that this is sharp.

Recently, Gr\"{u}nrock showed in \cite{grunrock:wave} that one can eliminate the gap to the almost criticality for the equation \eqref{main_eq} in dimension $n=3$, by taking the initial data in the Fourier-Lebesgue spaces $\^{H}_s^r$. These spaces are defined as the closure of the set of Schwartz functions under the norm,
\[
 \|f\|_{\^{H}_s^r}=\|\<\xi\>^s\^{f}\|_{L_\xi^{r'}}, \qquad \frac{1}{r}+\frac{1}{r'}=1,
\]
where $\^{f}$ denotes the (spatial) Fourier transform of $f$, and\footnote{Alternatively, one can use $\<\xi\>=1+|\xi|\simeq \sqrt{1+|\xi|^2}$.} $\<\xi\>=\sqrt{1+|\xi|^2}$. The norm of the corresponding homogeneous space is $\|f\|_{\dot{\^{H}}_s^r}=\||\xi|^s\^{f}\|_{L_\xi^{r'}}$. The norms of the initial data in the homogeneous Fourier-Lebesgue spaces scale as
\[
  \|u_\lambda(0, \cdot)\|_{\dot{\^{H}}_s^r(\R^n)}=\lambda^{s-\frac{n}{r}}\|u(0, \cdot)\|_{\dot{\^{H}}_s^r(\R^n)},
\]
so the critical exponent for these spaces is $s_c^r=\frac{n}{r}$. In terms of scaling of the norms, we therefore observe the following correspondence between the homogeneous Sobolev and Fourier-Lebesgue spaces,
\begin{equation}\label{scaling_cor}
 \dot{\^{H}}_s^r\sim \dot{H}^\sigma, \qquad \text{if} \qquad \sigma=s+n\left (\frac{1}{2}-\frac{1}{r}\right ).
\end{equation}
Gr\"{u}nrock's result established LWP for data in the space $\^{H}_s^r$ for $s>\frac{2}{r}+1$, $1<r\leq 2$. This range of exponents almost reaches a critical pair at the endpoint $r=1$. He relies on free wave interaction estimates of Foschi-Klainerman \cite{foschi-klainerman:bilinear}, which have a factor of $||\tau|-|\xi||^{\frac{n-3}{2}}$. This factor becomes unbounded near the null cone $|\tau|=|\xi|$ in dimension $n=2$. Thus, Gr\"{u}nrock's result does not directly generalize to dimension $n=2$. However, if this factor can be offset by cancellations in the nonlinearity along the null-cone, then his arguments can be salvaged. This approach was explored by the first author and A. Nahmod in \cite{grig-nahmod}, who achieve almost critical LWP for the null-form problem in 2D.

The main result of this paper is the following.

\begin{theorem}\label{main_thm}
 Let $1<r\leq 2$, $s>\frac{3}{2r}+1$, then the Cauchy problem \eqref{main_eq}-\eqref{main_IC} is locally well-posed for data in the space $\^{H}_s^r\times \^{H}_{s-1}^r$. 
\end{theorem}

\begin{remark}
 Well-posedness here is understood in the sense of Theorem \ref{gen_LWP}, with the solution found in the space $Z_{s, b}^r$, which will be defined in the next section.
\end{remark}

\begin{remark}
 At one extreme, $(s, r)=(\frac{7}{4}, 2)$, our result coincides with the local well-posedness for data in $H^{\frac{7}{4}+}$, while at the other extreme, $(s, r)=(\frac{5}{2}, 1)$, we obtain local well-posedness in the space $\^{H}_{\frac{5}{2}}^{1+}$. By the scaling correspondence \eqref{scaling_cor}, this last space scales like the Sobolev space $H^{\frac{3}{2}+}$, which gives a $\frac{1}{4}$ derivative improvement on the $H^{\frac{7}{4}+}$ Sobolev data result.
\end{remark}

\begin{remark}\label{other_prob}
 An analogous result to Theorem \ref{main_thm} can be proved for the equation
\begin{equation}\label{du2}
 \Box u=\del(u^2).
\end{equation}
This equation is invariant under the scaling $u\mapsto u_\lambda$, where $u_\lambda(t, x)=\lambda u(\lambda t, \lambda x)$. And the critical exponent on the Sobolev scale is $s_c=\frac{n}{2}-1$, while on the Fourier-Lebesgue scale it is $s_c^r=\frac{n}{r}-1$. The best known result for \eqref{du2} in dimension $n=2$ for data in $H^s$ is for $s>\frac{3}{4}$, which, again, can be shown using Strichartz estimates. A similar argument to the one we use for equation \eqref{main_eq} will show that the Cauchy problem for \eqref{du2} with data in $\^{H}_s^r$ is locally well-posed for $s>\frac{3}{2r}$, $1<r\leq 2$.
\end{remark}

The region for the parameters $(s, \frac{1}{r})$, for which Theorem \ref{main_thm} holds is shaded in Figure 1. Notice that the bottom and right edges of the region are not included. The $s$-line $s=\frac{3}{2r}+1$ almost reaches the point $(s, r)=(\frac{5}{2}, 1)$, which corresponds to the Sobolev space $H^{\frac{3}{2}}$. This is still above the critical regularity, as can be seen from the picture, where the solid line represents the critical relation $s_c^r=\frac{2}{r}$.

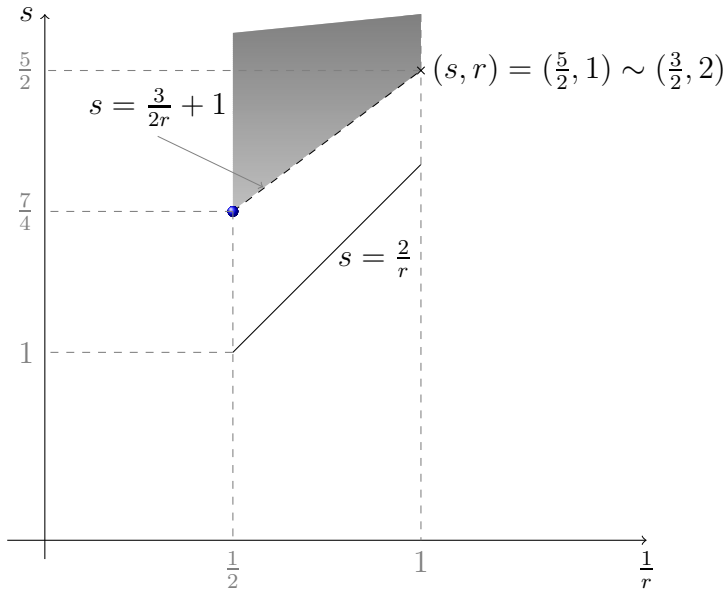
\begin{figure}[h]
\begin{tikzpicture}[x=5cm, y=2.5cm]
\draw[->] (-0.1,0) -- (1.6,0) node[below] {$\frac{1}{r}$};
\draw[->] (0,-0.1) -- (0,2.8) node[left] {$s$};
\shade[top color=gray,bottom color=gray!50]
      (0.5,1.755) -- (0.5, 2.7) -- (1, 2.8) -- (1, 2.5) -- cycle;
\draw[dashed, domain=0.5:1] plot (\x,1.5*\x+1);
\draw[domain=0.5:1] plot (\x,2*\x);

\draw (0.3, 2.15) node[above] {$s=\frac{3}{2r}+1$};
\draw[->, very thin, color=gray] (0.3, 2.15) -- (0.58, 1.88); 

\draw (0.75, 1.5) node[right] {$s=\frac{2}{r}$};

\draw plot[only marks,mark=ball, mark options={color=black}] coordinates{(0.5,1.75)};

\draw[color=gray, dashed] (0.5,1.75) -- (0.5,0) node[below] {$\frac{1}{2}$};
\draw[color=gray, dashed] (1,2.8) -- (1,0)  node[below] {1};
\draw[color=gray, dashed] (0.5,1) -- (0,1)  node[left] {$1$};
\draw[color=gray, dashed] (0.5,1.75) -- (0,1.75)  node[left] {$\frac{7}{4}$};
\draw[color=gray, dashed] (1,2.5) -- (0,2.5)  node[left] {$\frac{5}{2}$};

\draw[color=black] plot[only marks,mark=x] coordinates{(1,2.5)};
\draw (1, 2.5) node[right] {$(s, r)=(\frac{5}{2}, 1) \sim (\frac{3}{2}, 2)$};
\end{tikzpicture}
\caption{The shaded region represents the range of indeces for which LWP for data in space $\^{H}_s^r\times \^{H}_{s-1}^r$ holds.}
\end{figure}

The rest of the paper is organized as follows. in Section \ref{sec2} we introduce the solution spaces, and reduce Theorem \ref{main_thm} to a bilinear estimate. This reduction is achieved by utilizing the general LWP theorem stated in Appendix \ref{appB}. In Section \ref{sec3} we establish bilinear Fourier restriction estimates in $\^{L}^r$ spaces. These estimates are used in Section \ref{sec4} to establish the main bilinear estimate \eqref{key}.

As was mentioned earlier, we also sketch the proof of the LWP for equation \eqref{main_eq} with data in the Sobolev space $H^{\frac{7}{4}+}\times H^{\frac{3}{4}+}$ by the Strichartz estimates in Appendix \ref{appA}.

\vspace{10pt}
\textbf{Acknowledgments:} The authors would like to thank Andrea Nahmod for helpful discussions and her encouragement during this project.

The first author would also like to thank Magdalena Czubak for showing him the proof of the $H^{\frac{7}{4}+}$ LWP via the Strichartz estimates, which is sketched in Appendix \ref{appA}.

We also thank the anonymous referee for numerous helpful suggestions.


\section{The $X_{s, b}^r$ space, reduction to a bilinear estimate}\label{sec2}
We will obtain the local in time solution via a contraction principle in the time restricted $X_{s, b}^r$ space. This space is a Fourier-$L^{r'}$ analogue of the wave-Sobolev space $X^{s, b}$. It is defined by its norm
\[
 \|u\|_{X_{s, b}^r}=\|\<\xi\>^s\<|\tau|-|\xi|\>^b\~{u}\|_{L_{\tau, \xi}^{r'}},
\]
where $\~{u}$ stands for the time-space Fourier transform of $u$. The time restricted space is defined as
\[
 X_{s, b; T}^r=\left \{u=U|_{[-T, T]\times \R^n} : U\in X_{s, b}^r\right \},
\]
with its norm given by
\[
 \|u\|_{X_{s, b; T}^r}=\inf \left \{\|U\|_{X_{s, b}^r} : U|_{[-T, T]\times \R^n}=u \right \}.
\]

If $b>\frac{1}{r}$, the so-called transfer principle allows one to transfer multilinear estimates for free solutions of the homogeneous equation $\Box u=0$ with data in $\^{H}_s^r$ spaces to general elements of $X_{s, b}^r$ spaces. A trivial consequence of the transfer principle is the following crucial embedding
\[
 X_{s, b}^r\subset C(\R, \^{H}_s^r), \qquad \text{thus, also} \qquad X_{s, b; T}^r\subset C([-T, T], \^{H}_s^r),
\]
which guarantees that the solutions found in the $X_{s, b}^r$ spaces are the proper continuations of the initial data in $\^{H}_s^r$ spaces. For details on the transfer principle see \cite{selberg_thesis} for $L^2$ based spaces, \cite{grunrock:mkdv} for the linear case on the $X_{s, b}^r$ spaces, and \cite{grig-nahmod} for the $X_{s, b}^r$ multilinear case.

We need to separately estimate the time derivative of the solution, since the wave operator is of second order in time. For this, we define our solution space, $Z_{s, b}^r$, by its norm
\[
 \|u\|_{Z_{s, b}^r}=\|u\|_{X_{s, b}^r}+\|\del_t u\|_{X_{s-1, b}^r}.
\]
The time restricted space $Z_{s, b; T}^r$ and its norm are defined similar to the $X_{s, b; T}^r$ space above.

We will also use the simplified notation $\^{L}^r_{t, x}=X_{0, 0}^r$, and $\^{L}^r=\^{H}_0^r$, with the last norm taken either with respect to the time, or space variables.

Relying on the general local well-posedness Theorem \ref{gen_LWP}, we will prove Theorem \ref{main_thm} by establishing the following two estimates
\begin{equation}\label{N_X}
 \|\del u \del v\|_{X_{s-1, b+\epsilon-1}^r}\lesssim \|u\|_{Z_{s, b}^r}\|v\|_{Z_{s, b}^r},
\end{equation}
and
\begin{equation}\label{N_X_cont}
 \|(\del u)^2- (\del v)^2\|_{X_{s-1, b+\epsilon-1}^r}\lesssim \left (\|u\|_{Z_{s, b}^r}+\|v\|_{Z_{s, b}^r}\right )\|u-v\|_{Z_{s, b}^r}.
\end{equation}
The second estimate immediately follows from the first one due to the quadratic nature of the nonlinearity. Thus, we only need to prove estimate \eqref{N_X}.

If $b+\epsilon-1< 0$, it is easy to see that estimate \eqref{N_X} will follow from the estimate
\begin{equation}\label{key}
 \|uv\|_{X_{\sigma, 0}^r}\lesssim \|u\|_{X_{\sigma, b}^r}\|v\|_{X_{\sigma, b}^r},
\end{equation}
where $\sigma=s-1$, and $u, v$ are now general elements of $X_{\sigma, b}^r$. The rest of this paper is dedicated to proving \eqref{key} for $\sigma>\frac{3}{2r}$, $1<r\leq 2$, and some $b, \epsilon$, with $\frac{1}{r}<b<1$, $\epsilon\in (0, 1-b)$.

\begin{remark}
 In \cite{grunrock:wave} Gr\"{u}nrock relies on the general local well-posedness theorem for the first order equation $\del_t u-i\phi(D)u=N(u)$, which he proved in the earlier paper \cite{grunrock:mkdv}. For this, he reformulates equation \eqref{main_eq} as a system of first order equations for $u_\pm=u\pm i(1-\Delta_x)^{-\frac{1}{2}}\del_t u$. If we were to use this approach in our context, and the nonlinearity contains time derivatives of $u$, we would have to place the left hand sides of estimates \eqref{N_X}-\eqref{N_X_cont} in the space $Z_{s-1, b+\epsilon-1}^r$. These estimates in addition to \eqref{key} would also require the estimate
\begin{equation}\label{key2}
 \|uv\|_{X_{\sigma-1, b+\epsilon}^r}\lesssim \|u\|_{X_{\sigma, b}}\|v\|_{X_{\sigma, b}}.
\end{equation}
This is due to the fact that the absolute value of the symbol of the first order operator, $|\tau-\phi(\xi)|$, does not control $|\tau|$. However, Gr\"{u}nrock uses the transfer principle and does his estimates only on free solutions, which allows an easy control of the time derivatives by the spatial derivatives, and essentially produces the three dimensional analogue of \eqref{key2} for free from \eqref{key}. As we mentioned before, this approach fails in two dimensions. Nevertheless, one can indeed show \eqref{key2} with the arguments we employ here and a hyperbolic Leibniz type rule. See \cite{tanguay-thesis} for the details, where also more general estimates of the form \eqref{N_X} are proved.
\end{remark}


\section{Bilinear Fourier restriction estimates}\label{sec3}
By duality, we can reformulate the multiplicative estimate \eqref{key} as the trilinear integral estimate
\begin{equation}\label{key_dual}
 I\lesssim \|F_0\|_{L^r_{\tau, \xi}}\|F_1\|_{L^p_{\tau, \xi}}\|F_2\|_{L^p_{\tau, \xi}},
\end{equation}
where
\begin{equation*}
 I=\iiint \frac{F_0(X_0)F_1(X_1)F_2(X_2)\delta(X_0+X_1+X_2)\, dX_0 dX_1 dX_2}{\<\xi_0\>^{-\sigma}\<\xi_1\>^{\sigma}\<\xi_2\>^{\sigma}\<|\tau_1|-|\xi_1|\>^b\<|\tau_2|-|\xi_2|\>^b},
\end{equation*}
with $X_j=(\tau_j, \xi_j)$ for $j=0, 1, 2$; $F_1$, $F_2$ stand for
\begin{align*}
 F_1(\tau_1, \xi_1) & =\<\xi_1\>^\sigma\<|\tau_1|-|\xi_1|\>^b\~{u}(\tau_1, \xi_1),\\
 F_2(\tau_2, \xi_2) & =\<\xi_2\>^\sigma\<|\tau_2|-|\xi_2|\>^b\~{v}(\tau_2, \xi_2);
\end{align*}
and, to simplify notation, we use $p=r'$.

Since $\xi_0+\xi_1+\xi_2=0$ in the above integral, by triangle inequality, $|\xi_j|\leq |\xi_k|+|\xi_l|$ for all permutations $(j, k, l)$ of $(0, 1, 2)$. Hence, also, $\<\xi_j\>\lesssim \<\xi_k\>+\<\xi_l\>$. This implies that only low-high-high (LHH), high-low-high (HLH), and high-high-low (HHL) interactions are permitted. We then split $I=I_{\text{LHH}}+I_{\text{HLH}}+I_{\text{HHL}}$, with the following correspondence:
\begin{align*}
 \text{LHH} \quad & \leftrightarrow \quad \<\xi_0\> \lesssim \<\xi_1\>\sim \<\xi_2\>\\
 \text{HLH} \quad & \leftrightarrow \quad \<\xi_1\> \lesssim \<\xi_0\>\sim \<\xi_2\>\\
 \text{HHL} \quad & \leftrightarrow \quad \<\xi_2\> \lesssim \<\xi_0\>\sim \<\xi_1\>.
\end{align*}

The letters $M, N, L$ and their indexed counterparts will denote dyadic numbers of the form $2^j, j\in\{0, 1, 2, \dots\}$, and we will use the following notation for Fourier restrictions.
\begin{align*}
 F^N(X) & =\chi_{\<\xi\>\sim N}F(X),\\
 F^{N, L}(X) & =\chi_{\<|\tau|-|\xi|\>\sim L}F^N(X),\\
 F^{N, L, \pm}(X) & = \chi_{\pm \tau\geq 0} F^{N, L}(X).
\end{align*}
One immediately has the following dyadic summation estimates:
\begin{align*}
 \sum_{N}\|F^N\|_{L^p}^p & \sim \|F\|_{L^p}^p,\\
 \sum_{L}\|F^{N, L}\|_{L^p}^p & \sim \|F^N\|_{L^p}^p,\\
 \sum_{L}\|F^{N, L, \pm}\|_{L^p}^p & \lesssim \|F^N\|_{L^p}^p.
\end{align*}

Using the trilinear convolution form 
\begin{equation}\label{J}
 J(F_0, F_1, F_2)=\iiint F_0(X_0)F_1(X_1)F_2(X_2)\delta(X_0+X_1+X_2)\, dX_0 dX_1 dX_2,
\end{equation}
we have
\begin{equation}\label{I_localized}
 I\lesssim \sum_{\bf N, L}\frac{J(F_0^{N_0}, F_1^{N_1, L_1}, F_2^{N_2, L_2})}{N_0^{-\sigma}N_1^\sigma N_2^\sigma L_1^b L_2^b},
\end{equation}
where ${\bf N}=(N_0, N_1, N_2)$, ${\bf L}=(L_1, L_2)$. The proof of estimate \eqref{key} will rely on appropriate bounds for $J(F_0^{N_0}, F_1^{N_1, L_1}, F_2^{N_2, L_2})$ by $\|F_0^{N_0}\|_{L^r}\|F_1^{N_1, L_1}\|_{L^p}\|F_2^{N_2, L_2}\|_{L^p}$, with the constants depending on ${\bf N}$ and ${\bf L}$ in a way that allows summing on the right hand side in the last estimate, and, thus, establishing \eqref{key_dual}. For this, we need bilinear Fourier restriction estimates on $\R^{1+2}$ of the form
\begin{equation}\label{bi_triv}
 \|\P_{A_0}(\P_{A_1}u_1\cdot \P_{A_2}u_2)\|_{\^{L}^r}\leq C\|\P_{A_1}u_1\|_{\^{L}^r}\|\P_{A_2}u_2\|_{\^{L}^r},
\end{equation}
where $A_0, A_1, A_2\subset\R^{1+2}$ are measurable sets, and the projection $\P_A$ is defined by $\~{\P_A u}=\chi_A\~{u}$. We are interested in such restriction estimates with $A_0, A_1$ and $A_2$ being thickened subsets of the null cone $K=\{(\tau, \xi)\in\R^{1+2}\,:\,|\tau|=|\xi|\}$. Following Selberg \cite{selberg_bil_3d, selberg_bil_2d}, we introduce the following notation for thickened upper ($\tau\geq 0$) and lower ($\tau\leq 0$) cones, truncated in the spatial frequency $\xi$ by balls, annuli and angular sectors:
\begin{align*}
 K_{N, L}^\pm & = \{(\tau, \xi)\in\R^{1+2}\,:\, |\xi|\lesssim N, \tau=\pm|\xi|+O(L)\},\\
 \dot{K}_{N, L}^\pm & = \{(\tau, \xi)\in\R^{1+2}\,:\, |\xi|\sim N, \tau=\pm|\xi|+O(L)\},\\
 \dot{K}_{N, L, \gamma}^\pm(\omega) & = \{(\tau, \xi)\in\dot{K}_{N, L}^\pm\,:\,\theta(\pm\xi, \omega)\leq \gamma\},
\end{align*}
where $N, L, \gamma>0$; $\omega$ is an element of the unit circle, $\omega\in\S^1\subset \R^2$; and $\theta(a, b)$ denotes the angle between nonzero $a, b\in\R^2$. We will also use
\begin{align*}
 K_{N} & = \{(\tau, \xi)\in\R^{1+2}\,:\, |\xi|\lesssim N\},\\
 \dot{K}_{N} & = \{(\tau, \xi)\in\R^{1+2}\,:\, |\xi|\sim N\},\\
 \dot{K}_{N, \gamma}(\omega) & = \{(\tau, \xi)\in\dot{K}_{N}\,:\,\theta(\xi, \omega)\leq \gamma\}.
\end{align*}

By duality, \eqref{bi_triv} is equivalent to
\begin{equation}\label{dual_form}
 J(F_0^{-A_0}, F_1^{A_1}, F_2^{A_2})\leq C\|F_0^{-A_0}\|_{L^r}\|F_1^{A_1}\|_{L^p}\|F_2^{A_2}\|_{L^p},
\end{equation}
where $F^A(X)=\chi_A(X)F(X)$; $F_j=\~{u}_j, j=1, 2$; and $J$ is given by \eqref{J}.

We start with the following elementary lemma, which is a direct extension of \cite[Lemma 1.1]{selberg_bil_2d} to $\^{L}^r$ (see also \cite[Lemma 3.1]{selberg_bil_3d}). We use $|E|$ for the volume of the set $E\subset\R^{1+2}$.
\begin{lemma}\label{rest_triv}
 The estimate \eqref{bi_triv} holds with $C$ equal to an absolute constant times
\begin{align*}
 \min \bigg{(}\sup_{X\in A_0}|A_1\cap (X-A_2)|^{\frac{1}{r}}, & \sup_{X\in A_1}|A_0\cap (X+A_2)|^{\frac{1}{p}}|A_2|^{\frac{1}{r}-\frac{1}{p}},\\
& \qquad \sup_{X\in A_2}|A_0\cap (X+A_1)|^{\frac{1}{p}}|A_1|^{\frac{1}{r}-\frac{1}{p}}\bigg{)},
\end{align*}
provided this quantity is finite.
\end{lemma}

\begin{proof}
 For the first bound, we use the dual formulation, and rewrite
\begin{align*}
 & J(F_0^{-A_0}, F_1^{A_1}, F_2^{A_2})\\
 &\qquad\qquad=\int_{X_0}\int_{X_1} \chi_{X_1\in A_1\cap(X_0-A_2)}\chi_{X_0\in A_0}F_1^{A_1}(X_1)F_2^{A_2}(X_0-X_1)F_0^{-A_0}(-X_0)\, dX_1 dX_0\\
 & \qquad\qquad\quad \lesssim \int_{X_0} \left [\sup_{X_0\in A_0}|A_1\cap(X_0-A_2)|^{\frac{1}{r}}\right ]\left [\int_{X_1} \left (F_1^{A_1}(X_1)F_2^{A_2}(X_0-X_2)\right )^p\, dX_1\right ]^{\frac{1}{p}}\\
 &\hspace{4.6in}\times F_0^{-A_0}(-X_0)\, dX_0,
\end{align*}
where we used H\"{o}lder's inequality in the $X_1$ variable. Now applying H\"{o}lder's inequality in the $X_0$ variable and using Fubini's theorem, we bound the above by
\begin{align*}
 & J(F_0^{-A_0}, F_1^{A_1}, F_2^{A_2}) \lesssim \sup_{X_0\in A_0}|A_1\cap(X_0-A_2)|^{\frac{1}{r}}\\
 &\quad \times \left [\int_{X_0}\int_{X_1} \left (F_1^{A_1}(X_1)\right )^p \left (F_2^{A_2}(X_0-X_1)\right )^p\, dX_1 dX_0\right ]^{\frac{1}{p}}\left [\int_{X_0} \left (F_0^{-A_0}(-X_0)\right )^{r}\, dX_0\right ]^{\frac{1}{r}}\\
 & \qquad\qquad\lesssim \sup_{X_0\in A_0}|A_1\cap(X_0-A_2)|^{\frac{1}{r}} \\
 &\qquad\quad\qquad\times \left [\int_{X_1} \left (F_1^{A_1}(X_1)\right )^p \int_{X_0} \left (F_2^{A_2}(X_0-X_1)\right )^p\, dX_0 dX_1\right ]^{\frac{1}{p}} \|F_0^{-A_0}\|_{L^{r}}\\
 &\qquad\quad\qquad\qquad\lesssim \sup_{X_0\in A_0}|A_1\cap(X_0-A_2)|^{\frac{1}{r}} \|F_0^{-A_0}\|_{L^{r}}\|F_1^{A_1}\|_{L^p}\|F_2^{A_2}\|_{L^p}.
\end{align*}

The proofs of the second and third bounds are similar, so we prove only the second bound. Using the dual formulation and H\"{o}lder's inequality in $X_0$, we have
\begin{align*}
 & J(F_0^{-A_0}, F_1^{A_1}, F_2^{A_2})\\
 & \qquad\qquad=\int_{X_1}\int_{X_0} \chi_{X_0\in A_0\cap(X_1+A_2)}\chi_{X_1\in A_1}F_1^{A_1}(X_1)F_2^{A_2}(X_0-X_1)F_0^{-A_0}(-X_0)\, dX_1 dX_0\\
 & \qquad\qquad \lesssim \int_{X_1} \left [\sup_{X_1\in A_1}|A_0\cap(X_1+A_2)|^{\frac{1}{p}}\right ]\left [\int_{X_0} \left (F_0^{-A_0}(-X_0)F_2^{A_2}(X_0-X_1)\right )^{r}\, dX_0\right ]^{\frac{1}{r}}\\
 &\hspace{4.8in}\times F_1^{A_1}(X_1)\, dX_1.
\end{align*}
 Now applying H\"{o}lder's inequality in the $X_1$ variable gives
\begin{align*}
 & J(F_0^{-A_0}, F_1^{A_1}, F_2^{A_2}) \lesssim \sup_{X_1\in A_1}|A_0\cap(X_1+A_2)|^{\frac{1}{p}}\\
 & \qquad\qquad\qquad\times\left [\int_{X_1}\int_{X_0} \left (F_0^{-A_0}(-X_0)\right )^{r} \left (F_2^{A_2}(X_0-X_1)\right )^{r}\, dX_0 dX_1\right ]^{\frac{1}{r}}\|F_1^{A_1}\|_{L^{p}}\\
 &\qquad\qquad\lesssim \sup_{X_1\in A_1}|A_0\cap(X_1+A_2)|^{\frac{1}{p}} \|F_0^{-A_0}\|_{L^{r}}\|F_2^{A_2}\|_{L^{r}}\|F_1^{A_1}\|_{L^p}\\
 &\qquad\qquad\qquad \lesssim\sup_{X_1\in A_1}|A_0\cap(X_1+A_2)|^{\frac{1}{p}} |A_2|^{\frac{1}{r}-\frac{1}{p}}\|F_0^{-A_0}\|_{L^{r}}\|F_1^{A_1}\|_{L^p}\|F_2^{A_2}\|_{L^p},
\end{align*}
where in the last step we used the fact that $r<p$ and H\"{o}lder's inequality to bound $\|F_2^{A_2}\|_{L^{r}}\lesssim |A_2|^{\frac{1}{r}-\frac{1}{p}}\|F_2^{A_2}\|_{L^p}$.
\end{proof}

Using the above lemma, we can now prove bilinear restriction estimates in $\^{L}^r$, analogous to those established by Selberg in \cite[Theorem 2.1]{selberg_bil_2d}. In the sequel we make use of the notation $N_{\min}^{012}=\min\{N_0, N_1, N_2\}$, and similarly for $N_{\min}^{12}$, $L_{\min}^{12}$ and $L_{\max}^{12}$.

\begin{prop}\label{rest}
 Let $1<r\leq 2$, then the estimate
\begin{equation}\label{bi_rest}
 \left \|\P_{K_{N_0}}\left (\P_{K_{N_1, L_1}^{\pm_1}}u_1\cdot \P_{K_{N_2, L_2}^{\pm_2}}u_2\right )\right \|_{\^{L}^r}\leq C\|u_1\|_{\^{L}^r}\|u_2\|_{\^{L}^r}
\end{equation}
holds with
\begin{align}
 C & \sim (N_{\min}^{012})^{\frac{2}{p}}(N_{\min}^{12})^{\frac{2}{r}-\frac{2}{p}}(L_{\min}^{12})^{\frac{1}{r}}, \label{bi_easy}\\
 C & \sim (N_{\min}^{012})^{\frac{1}{p}}(N_{\min}^{12})^{\frac{3}{2r}-\frac{1}{p}}(L_{\min}^{12})^{\frac{1}{r}}(L_{\max}^{12})^{\frac{1}{2r}} \label{bi_hard},
\end{align}
regardless of the choice of signs $\pm_j$.
\end{prop}

\begin{remark}
 The proof of estimate \eqref{key} uses only the following weaker version of the above estimate with \eqref{bi_hard},
\begin{equation}\label{bi_rest_nec}
  \left \|\P_{\dot{K}_{N_1, L_1}^{\pm_1}}u_1\cdot \P_{\dot{K}_{N_2, L_2}^{\pm_2}}u_2\right \|_{\^{L}^r}\leq (N_{\min}^{12})^{\frac{3}{2r}}(L_{\min}^{12})^{\frac{1}{r}}(L_{\max}^{12})^{\frac{1}{2r}}\|u_1\|_{\^{L}^r}\|u_2\|_{\^{L}^r}.
\end{equation}
However, the proof of Proposition \ref{rest} is not much harder than what is required to prove the last estimate, hence our choice of stating and proving the estimates in the more general form for possible future uses.
\end{remark}

\begin{remark}\label{bi_easy_gen}
 The bound with \eqref{bi_easy} can be trivially generalized to any dimension, since Lemma \ref{rest_triv}, which is the only ingredient in its proof, holds irrespective of the dimension of the space. In $n$ dimensions \eqref{bi_easy} looks like
\begin{equation*}\label{bi_easy_n}
  C \sim (N_{\min}^{012})^{\frac{n}{p}}(N_{\min}^{12})^{\frac{n}{r}-\frac{n}{p}}(L_{\min}^{12})^{\frac{1}{r}}.\\
\end{equation*}
\end{remark}

\begin{proof}[Proof of \eqref{bi_easy}]
 The bound with \eqref{bi_easy} immediately follows from Lemma \ref{rest_triv}. To see this, we split into the cases $N_0\lesssim N_1\sim N_2$ (LHH), $N_1\lesssim N_0\sim N_2$ (HLH), and $N_2\lesssim N_0\sim N_1$ (HHL), and by symmetry consider only LHH and HLH cases.

In the HLH case, $N_1\lesssim N_2\sim N_0$, we estimate the volume of the set
\begin{equation*}
 E=K_{N_1, L_1}^{\pm_1}\cap(X_0-K_{N_2, L_2}^{\pm_2})\subset\{\xi_1\lesssim N_1; \tau_1=\pm_1|\xi_1|+O(L_1), \tau_0-\tau_1=\pm_2|\xi_0-\xi_1|+O(L_2)\},
\end{equation*}
for any $X_0\in K_{N_0}$. Integrating first in $\tau_1$ and then in $\xi_1$, we obtain $|E|\lesssim N_1^2L_{\min}^{12}$. Raising to the power $1/r$ and using the first bound from Lemma \ref{rest_triv} gives the desired result.

For the LHH regime, i.e. when $N_0\ll N_1\sim N_2$, due to symmetry we can assume $L_1\leq L_2$, and estimate the volume of the set
\begin{equation*}
 E=K_{N_0}\cap(X_2+K_{N_1, L_1}^{\pm_1})\subset \{\xi_0\lesssim N_0; \tau_0-\tau_2=\pm_1 |\xi_0-\xi_2|+O(L_1)|\},
\end{equation*}
for any $X_2\in K_{N_2, L_2}^{\pm_2}$. Again, integrating first in $\tau_0$ and then in $\xi_0$, we have $|E|\leq N_0^2 L_1$, which coupled with $|K_{N_1, L_1}^{\pm_1}|\leq N_1^2L_1$ and the third bound in Lemma \ref{rest_triv} gives the desired result.
\end{proof}

\begin{proof}[Proof of \eqref{bi_hard}]
 We may assume $L_{\max}^{12}\ll N_{\min}^{12}$, since otherwise \eqref{bi_easy} is already better than \eqref{bi_hard}. Additionally, we may replace $K_{N_j, L_j}^{\pm_j}$ by $\dot{K}_{N_j, L_j}^{\pm_j}$ due to the presence of the factor $(N_{\min}^{012})^{\frac{1}{p}}(N_{\min}^{12})^{\frac{3}{2r}-\frac{1}{p}}$ in the bound. Indeed, using the dual formulation and decomposing the balls $|\xi_j|\lesssim N_j$ into annuli $|\xi_j|\sim M_j$ for dyadic $0<M_j\leq N_j$, we can sum over $M_{\max}^{012}\leq N_{\min}^{012}$ using the factor $(M_{\min}^{012})^{\frac{1}{p}}$ and the fact that the two largest of the $M_j$'s are comparable. For the rest of the sum we use the duality of $l^p$ and $l^{r}$ in the HLH and HHL cases, or the factor $(M_{\min}^{12})^{\frac{3}{2r}-\frac{1}{p}}$ in the LHH case. We now proceed with the proof separately in the HLH and LHH cases. In what follows, we use the calculations and ideas of \cite[Section 3]{selberg_bil_2d} in a substantial way.

\noindent$\bullet$\quad\textbf{The HLH case: } $N_1\lesssim N_0\sim N_2$. By Lemma \ref{rest_triv} we reduce to estimating the volume of the set
\begin{equation*}
 E=\dot{K}_{N_1, L_1}^{\pm_1}\cap(X_0-\dot{K}_{N_2, L_2}^{\pm_2})
\end{equation*}
uniformly in $X_0\in \dot{K}_{N_0}$. Selberg estimated the volume of this set in the HLH regime in \cite[Subsection 3.1]{selberg_bil_2d} as $|E|\lesssim N_1^{\frac{3}{2}}L_{\min}^{12}(L_{\max}^{12})^{\frac{1}{2}}$, using the fact that $L_{\max}^{12}\ll N_1\lesssim N_2$ and relying on an estimate for the area of intersection of two thickened circles \cite[Lemma 2.1]{selberg_bil_2d}. Using this bound, we then have
\begin{align*}
 |E|^{\frac{1}{r}}\lesssim \left (N_1^{\frac{3}{2}} L_{\min}^{12}(L_{\max}^{12})^{\frac{1}{2}}\right )^{\frac{1}{r}} & \lesssim N_1^{\frac{3}{2r}}(L_{\min}^{12})^{\frac{1}{r}}(L_{\max}^{12})^{\frac{1}{2r}},
\end{align*}
which establishes \eqref{bi_hard} in the HLH regime.

\noindent$\bullet$\quad\textbf{The LHH case: } $N_0\ll N_1\sim N_2$. We may assume $L_1\leq L_2$ by symmetry, and moreover,
\begin{equation}\label{gamma0<1}
L_1\leq L_2\ll (N_0/N_1)^{\frac{2r}{p}}N_1,
\end{equation}
since otherwise, \eqref{bi_easy} is better than \eqref{bi_hard}. We need to prove that \eqref{bi_rest} holds with
\begin{equation}\label{C_hard_LHH}
 C\sim N_0^{\frac{1}{p}}N_1^{\frac{3}{2r}-\frac{1}{p}}L_1^{\frac{1}{r}}L_2^{\frac{1}{2r}}.
\end{equation}

Recall that \eqref{bi_rest} is equivalent to the dual estimate \eqref{dual_form}. By a simple change of variables in the integral in the trilinear form on the left, we can also see that \eqref{dual_form} is equivalent to the product estimate
\begin{equation}\label{u_1u_0}
 \left \|\P_{K_{N_2, L_2}^{\pm_2}}\left (\P_{K_{N_1, L_1}^{\pm_1}}u_1\cdot \P_{K_{N_0}}u_0\right )\right \|_{\^{L}^p}\leq C\|u_1\|_{\^{L}^r}\|u_0\|_{\^{L}^p},
\end{equation}
where $u_0$ is the inverse Fourier transform of $F_0$. So it suffices to prove \eqref{u_1u_0} for $C$ given by \eqref{C_hard_LHH}.

For the proof of the last estimate, we need to further decompose the Fourier supports of $u_1$ and $u_0$ into angular sectors $\G_\gamma(\omega)=\{\xi\in \R^2 : \theta(\xi, \omega)\leq \gamma\}$. We list some of the well-known facts about such decompositions (c.f. \cite[Section 2.5]{selberg_bil_2d}).

Let $\Omega(\gamma)\subset \S^1$ be a maximal $\gamma$-separated subset of the unit circle $\S^1\subset \R^2$, where $0<\gamma<\pi$. Then
\begin{equation}\label{ang_part}
 1\leq \sum_{\omega\in\Omega(\gamma)} \chi_{\G_\gamma(\omega)}(\xi)\leq 5 \qquad (\text{for all } \xi\neq 0), 
\end{equation}
where the inequality on the left follows from maximality of $\Omega(\gamma)$, and the right inequality from the $\gamma$-separation, which results in the almost orthogonality condition
\begin{equation}\label{ang_ortho}
\#\{\omega'\in\Omega(\gamma) : \theta(\omega', \omega)\leq k\gamma\}\leq 2k+1,
\end{equation}
for any $k\in \N$ and $\omega\in\Omega(\gamma)$. We will employ the following notation for localizations to the angular sectors
\[
 u_1^{\gamma, \omega}=\P_{\theta(\pm_1\xi, \omega)\leq \gamma} u_1, \qquad u_0^{\gamma, \omega}=\P_{\theta(\xi, \omega)\leq \gamma} u_0.
\]
By \eqref{ang_part}
\[
 \|\~{u_1}\|_{L^p}^{p}\sim \sum_{\omega\in\Omega(\gamma)} \|\~{u_1^{\gamma, \omega}}\|_{L^p}^p, \qquad \|\~{u_0}\|_{L^r}^{r}\sim \sum_{\omega\in\Omega(\gamma)} \|\~{u_0^{\gamma, \omega}}\|_{L^r}^r.
\]
Hence, also
\begin{equation}\label{gammaSum}
 \sum_{\substack{\omega_1, \omega_0\in \Omega(\gamma) \\ \theta(\omega_1, \omega_0)\lesssim \gamma}} \|u_1^{\gamma, \omega_1}\|_{\^{L}^r} \|u_0^{\gamma, \omega_0}\|_{\^{L}^p}\lesssim \|u_1\|_{\^{L}^r}\|u_0\|_{\^{L}^p},
\end{equation}
where we used H\"{o}lder's inequality and \eqref{ang_ortho}.

We also observe the following Whitney decomposition with respect to angular variables, the proof of which is elementary and is omitted.

\begin{lemma}\label{whitney1}
 We have
\[
 1\sim \sum_{\substack{0<\gamma<1\\ \gamma - \text{dyadic}}} \; \sum_{\substack{\omega_1, \omega_0\in \Omega(\gamma)\\ 3\gamma \leq \theta(\omega_1, \omega_0)\leq 12\gamma}} \chi_{\G_\gamma(\omega_1)}(\xi_1) \chi_{\G_\gamma(\omega_0)}(\xi_0),
\]
for all $\xi_1, \xi_0\in \R^2\setminus \{0\}$, with $\theta(\xi_1, \xi_0)>0$.
\end{lemma}

Note that in the above decomposition $\G_\gamma(\omega_1)$ and $\G_\gamma(\omega_0)$ are always separated by an angle of at least $\gamma$, as implied by the condition $\theta(\omega_1, \omega_0)\geq 3\gamma$. If this separation is not needed, then one has the following.

\begin{lemma}\label{whitney2}
 For any $0<\gamma<1$, $k\in \N$,
\[
 \chi_{\theta(\xi_1, \xi_0)\leq k\gamma}\lesssim \sum_{\substack{\omega_1, \omega_0\in \Omega(\gamma)\\ \theta(\omega_1, \omega_0)\leq (k+2)\gamma}} \chi_{\G_\gamma(\omega_1)}(\xi_1) \chi_{\G_\gamma(\omega_0)}(\xi_0),
\]
for all $\xi_1, \xi_0\in \R^2\setminus \{0\}$.
\end{lemma}
We again skip the simple proof.

Returning to the proof of estimate \eqref{u_1u_0}, we denote $\theta_{01}=\theta(\xi_0, \pm_1\xi_1)$, and split into the cases $\theta_{01}\lesssim \gamma_0$, and $\theta_{01}\gg \gamma_0$, where
\begin{equation*}
 \gamma_0=\left (\frac{L_2}{N_1}\right )^{\frac{1}{2}}\left (\frac{N_1}{N_0}\right )^{\frac{r}{p}}.
\end{equation*}
This choice for $\gamma_0$ will be apparent from our bounds in the two cases.  Observe that $\gamma_0\ll 1$ due to \eqref{gamma0<1}.

Assuming without loss of generality that $\~{u}_1, \~{u}_0\geq 0$, we have by Lemmas \ref{whitney1} and \ref{whitney2},
\begin{align*}
\|\P_{N_2}(u_1u_0)\|_{\^{L}^p} & \lesssim \sum_{\omega_0, \omega_1\in\Omega(\gamma_0)} \chi_{\theta(\omega_0, \omega_1)\lesssim \gamma_0}\|\FT\P_{N_2}(u_1^{\gamma_0, \omega_1} u_0^{\gamma_0, \omega_0})\|_{L^r}\\
&\qquad +\sum_{\gamma_0<\gamma<1}\sum_{\omega_0, \omega_1\in\Omega(\gamma)} \chi_{3\gamma\leq \theta(\omega_0, \omega_1)\leq 12\gamma}\|\FT\P_{N_2}(u_1^{\gamma, \omega_1} u_0^{\gamma, \omega_0})\|_{L^r}=\Sigma_1+\Sigma_2,
\end{align*}
where $\FT$ denotes the time-space Fourier transform.

For $\Sigma_1$, we estimate the volume of the set $E=\dot{K}_{N_0, \gamma_0}(\omega_0)\cap(X_2+\dot{K}^{\pm_1}_{N_1, L_1, \gamma_0}(\omega_1))$, where $X_2$ is any point in $\dot{K}^{\pm_2}_{N_2, L_2}$. Notice that
\begin{equation*}
 E\subset \{(\tau_0, \xi_0) : |\xi_0|\lesssim N_0, \theta(\xi_0, \omega_0)<\gamma_0, \tau_0=\pm_1|\xi_0-\xi_2|+\tau_2+O(L_1)\}.
\end{equation*}
Integrating first in $\tau_0$, and then $\xi_0$, we have $|E|\lesssim N_0^2\gamma_0 L_1$. Further observing that $|A_1|\sim N_1^2\gamma_0 L_1$, we have by Lemma \ref{rest_triv} that
\begin{equation*}
 \|\FT\P_{N_2}(u_1^{\gamma_0, \omega_1} u_0^{\gamma_0, \omega_0})\|_{L^r}\lesssim (N_0^2\gamma_0 L_1)^{\frac{1}{p}}(N_1^2\gamma_0 L_1)^{\frac{1}{r}-\frac{1}{p}}\|\~{u_1^{\gamma_0, \omega_1}}\|_{L^p}\|\~{u_0^{\gamma_0, \omega_0}}\|_{L^r},
\end{equation*}
and hence,
\begin{align*}
 \Sigma_1 
 &\lesssim N_0^{\frac{2}{p}}N_1^{\frac{2}{r}-\frac{2}{p}}\gamma_0^{\frac{1}{r}} L_1^{\frac{1}{r}}\sum_{\omega_1, \omega_0\in\Omega(\gamma_0)}\chi_{\theta(\omega_1, \omega_0)\lesssim\gamma_0}\|\~{u_1^{\gamma_0, \omega_1}}\|_{L^p}\|\~{u_0^{\gamma_0, \omega_0}}\|_{L^r}\\
 &\qquad\qquad \lesssim N_0^{\frac{2}{p}}N_1^{\frac{2}{r}-\frac{2}{p}}\gamma_0^{\frac{1}{r}} L_1^{\frac{1}{r}}\|\~{u_1}\|_{L^p}\|\~{u_0}\|_{L^r},
\end{align*}
where in the last step we used \eqref{gammaSum}. Finally, using $\gamma_0=(L_2/N_1)^{\frac{1}{2}}(N_1/N_0)^{\frac{r}{p}}$ gives the desired estimate.

For $\Sigma_2$, we need to further decompose the Fourier supports of $u_1$ and $u_2$ into angular sectors. For this we use $\overline{\gamma}$ for the size of the sectors and $\overline{\omega}_1$, $\overline{\omega}_2$ for the centers of the sectors respectively, to distinguish them from those of the previous angular decompositions. Using the same Whitney-type decomposition as before, we have
\begin{align*}
\Sigma_2 & =\sum_{\gamma_0<\gamma<1}\sum_{\omega_0, \omega_1\in\Omega(\gamma)} \chi_{3\gamma\leq \theta(\omega_0, \omega_1)\leq 12\gamma} \sum_{\overline{\omega}_1, \overline{\omega}_2\in\Omega(\gamma_0)} \chi_{\theta(\overline{\omega}_1, \overline{\omega}_2)\lesssim \gamma_0}\|\FT\P_{N_2, L_2}^{\gamma_0, \overline{\omega}_2}(u_1^{\gamma, \omega_1, \gamma_0, \overline{\omega}_1} u_0^{\gamma, \omega_0})\|_{L^r}\\
&\quad +\sum_{\gamma_0<\gamma<1}\sum_{\omega_0, \omega_1\in\Omega(\gamma)} \chi_{3\gamma\leq \theta(\omega_0, \omega_1)\leq 12\gamma} \\
&\qquad\qquad\qquad \sum_{\gamma_0<\overline{\gamma}<1}\sum_{\overline{\omega}_1, \overline{\omega}_2\in\Omega(\overline{\gamma})} \chi_{3\overline{\gamma}\leq \theta(\overline{\omega}_1, \overline{\omega}_2)\leq 12\overline{\gamma}} \|\FT\P_{N_2, L_2}^{\overline{\gamma}, \overline{\omega}_2}(u_1^{\gamma, \omega_1, \overline{\gamma}, \overline{\omega}_1} u_0^{\gamma, \omega_0})\|_{L^r}.
\end{align*}
But since $\xi_0=\xi_1+\xi_2$, and, hence, $\theta(\xi_1, \xi_0)\leq \theta(\xi_1, \xi_2)$, the first term in the above sum is empty, and the $\overline{\gamma}$ in the second sum is restricted to $\gamma_0<\gamma\lesssim \overline{\gamma}<1$. Furthermore, since
\[
\#\{\overline{\omega}_1\in\Omega(\overline{\gamma}): \Gamma_\gamma(\omega_1)\cap \Gamma_{\overline{\gamma}}(\overline{\omega})\ne \emptyset\}=\mathcal{O}(1),
\]
the last sum in the second term in $\Sigma_2$ is trivial. We thus have
\begin{align*}
\Sigma_2 \lesssim \sum_{\gamma_0<\gamma<1} \sum_{\gamma<\overline{\gamma}<1} \sum_{\omega_0, \omega_1\in\Omega(\gamma)} \chi_{3\gamma\leq \theta(\omega_0, \omega_1)\leq 12\gamma}  \chi_{3\overline{\gamma}\leq \theta(\overline{\omega}_1, \overline{\omega}_2)\leq 12\overline{\gamma}} \|\FT\P_{N_2, L_2}^{\overline{\gamma}, \overline{\omega}_2}(u_1^{\gamma, \omega_1, \overline{\gamma}, \overline{\omega}_1} u_0^{\gamma, \omega_0})\|_{L^r}.
\end{align*}

To estimate the summand in the above sum, we use Lemma \ref{rest_triv} and an estimate for the volume of the set $E=\dot{K}^{\pm_1}_{N_1, L_1, \overline{\gamma}}(\overline{\omega}_1)\cap(X_0-\dot{K}^{\pm_2}_{N_2, L_2, \overline{\gamma}}(\overline{\omega}_2))$ established by Selberg in \cite[Section 3.3]{selberg_bil_2d}.\footnote{Selberg relies on self-duality of $L^2$ to further break the supports of $\widehat{u}_1$ and $\widehat{u}_2$ into balls of radius $N_0$ via a standard tiling argument, to obtain the estimate $|E_t|<N_0L_1L_2/\gamma$ for the set $E_t$ with the tiled supports. In our case the absence of self-duality would lead to losses in the tiling argument, however the estimate $|E|<N_1L_1L_2/\overline{\gamma}$ follows from Selberg's argument in \cite[Section 3.3]{selberg_bil_2d} verbatim without the tiling.} Relying on the separation of the angular sectors of $u_1$ and $u_2$, Selberg showed that the volume of this set can be estimated by $|E|<N_1L_1L_2/\overline{\gamma}$.

Using this estimate for the volume of $E$, $|E|<N_1L_1L_2/\overline{\gamma}$, we have from Lemma \ref{rest_triv},
\begin{equation*}
 \|\FT\P_{N_2, L_2}^{\overline{\gamma}, \overline{\omega}_2}(u_1^{\gamma, \omega_1, \overline{\gamma}, \overline{\omega}_1} u_0^{\gamma, \omega_0})\|_{L^r}\lesssim \left (\frac{N_1 L_1 L_2}{\overline{\gamma}}\right )^{\frac{1}{r}}\|\~{u_1^{\gamma, \omega_1}}\|_{L^p}\|\~{u_0^{\gamma, \omega_0}}\|_{L^r}.
\end{equation*}
Hence,
\begin{align*}
 \Sigma_2
 & \lesssim N_1^{\frac{1}{r}}L_1^{\frac{1}{r}}L_2^{\frac{1}{r}} \sum_{\gamma_0<\gamma<1} \sum_{\gamma<\overline{\gamma}<1} \overline{\gamma}^{-\frac{1}{r}}\sum_{\omega_0, \omega_1\in\Omega(\gamma)} \chi_{3\gamma\leq \theta(\omega_0, \omega_1)\leq 12\gamma}  \|\~{u_1^{\gamma, \omega_1}}\|_{L^p}\|\~{u_0^{\gamma, \omega_0}}\|_{L^r}\\
 & \qquad\qquad\lesssim N_1^{\frac{1}{r}}L_1^{\frac{1}{r}} L_2^{\frac{1}{r}} \gamma_0^{-\frac{1}{r}}\|\~{u_1}\|_{L^p}\|\~{u_0}\|_{L^r},
\end{align*}
where we estimated the sum in $\omega_1, \omega_0$ as in the case of $\Sigma_1$, and successively summed in $\overline{\gamma}$ then in $\gamma$. Using $\gamma_0=(L_2/N_1)^{\frac{1}{2}}(N_1/N_0)^{\frac{r}{p}}$ once again gives the desired estimate.
\end{proof}


\section{The proof of estimate \eqref{key}}\label{sec4}
We start by establishing some simple summation lemmas.
\begin{lemma}\label{summation} Let $A, B\in \R$, then
\begin{equation*}
 \sum_{N_0\lesssim N_1} N_0^A\lesssim N_1^B,
\end{equation*}
provided $(i)\; B \geq A$, $(ii)\; B\geq 0$ and $(iii)\; A=B=0$ is excluded.
\end{lemma}

\begin{proof}
The statement of the lemma follows immediately from the following trivial dyadic summation rule
\begin{equation}\label{Sigma_A}
 \sum_{N_0} \chi_{N_0\leq N_1} N_0^A\sim \left \{
\begin{array}{ll}
 N_1^A & \text{if } A>0,\\
 \log\<N_1\> & \text{if } A=0,\\
 1 & \text{if } A<0.
\end{array}\right . 
\end{equation}
\end{proof}

The next two dyadic summation lemmas will be used for estimating HLH (and by symmetry, HHL) and LHH interactions respectively.

\begin{lemma}\label{HLH_dyadic}
 Let $A, B\in\R$. The estimate
\begin{equation*}
 \sum_{\bf N} \chi_{N_1\leq N_0\sim N_2} \frac{N_1^A}{N_0^B} \|F_0^{N_0}\|_{L^{r}}\|F_1^{N_1}\|_{L^p}\|F_2^{N_2}\|_{L^p}\lesssim \|F_0\|_{L^{r}}\|F_1\|_{L^p}\|F_2\|_{L^p}
\end{equation*}
holds, provided that (i) $B\geq A$, (ii) $B\geq 0$ and (iii) $A=B=0$ is excluded.
\end{lemma}

\begin{proof}
 The proof of this lemma is contained in the proof of \cite[Lemma 2.1]{DFS_2d}, where, after applying Lemma \ref{summation}, one simply needs to use the duality of $l^p$ and $l^{r}$ (H\"{o}lder's inequality for $l^p$), instead of the Cauchy-Schwartz inequality.
\end{proof}

\begin{lemma}\label{LHH_dyadic}
 Let $A, B\in\R$. The estimate
\begin{equation*}
 \sum_{\bf N} \chi_{N_0\leq N_1\sim N_2} \frac{N_0^A}{N_1^B} \|F_0^{N_0}\|_{L^{r}}\|F_1^{N_1}\|_{L^p}\|F_2^{N_2}\|_{L^p}\lesssim \|F_0\|_{L^{r}}\|F_1\|_{L^p}\|F_2\|_{L^p}
\end{equation*}
holds, provided that (i) $B>A$, (ii) $B>0$.
\end{lemma}

\begin{proof}
 The proof is similar to the proof of the previous lemma, with the only difference in that one cannot sum in $N_1\sim N_2$ using H\"{o}lder's inequality, since $l^p$ is not self-dual. But due to the strict inequalities in the hypothesis, we have $A<B-\epsilon$ for some $\epsilon>0$, such that $B-\epsilon>0$. And hence, by Lemma \ref{summation},
\begin{align*}
 \sum_{\bf N} \chi_{N_0\leq N_1\sim N_2} \frac{N_0^A}{N_1^B} \|F_0^{N_0}\|_{L^{r}}\|F_1^{N_1}\|_{L^p}\|F_2^{N_2}\|_{L^p}
 & \lesssim \|F_0\|_{L^{r}}\sum_{N_1\sim N_2} N_1^{-\epsilon}\|F_1^{N_1}\|_{L^{p}}\|F_2^{N_2}\|_{L^p}\\
 & \lesssim \|F_0\|_{L^{r}} \|F_2\|_{L^{p}} \left (N_1^{-\epsilon}\right )_{l^{r}}\left (\|F_1^{N_1}\|_{L^p}\right )_{l^p}\\
 & \lesssim \|F_0\|_{L^{r}} \|F_1\|_{L^{p}}\|F_2\|_{L^{p}}
\end{align*}
\end{proof}

We are now ready to prove the key estimate \eqref{key}.

\begin{proof}[Proof of estimate \eqref{key}]
It is enough to prove the dual estimate \eqref{key_dual}. We will use dyadic decompositions, and by symmetry, can assume $L_1\leq L_2$. By \eqref{I_localized}, estimate \eqref{key_dual} then reduces to showing
\begin{equation*}\label{I_1}
 \sum_{\bf N} \frac{S_{\bf N}}{N_0^{-\sigma} N_1^\sigma N_2^\sigma}\lesssim \|F_0\|_{L^{r}}\|F_1\|_{L^p}\|F_2\|_{L^p},
\end{equation*}
where
\[
 S_N=\sum_{\bf L}\chi_{L_1\leq L_2}\frac{J(F_0^{N_0}, F_1^{N_1, L_1}, F_2^{N_2, L_2})}{L_1^b L_2^b},
\]
with $J$ again defined as in \eqref{J}. By symmetry, we only need to consider the HLH and LHH cases.

\noindent$\bullet$\quad\textbf{The HLH case: } $N_1\lesssim N_0\sim N_2$. From estimate \eqref{bi_rest_nec} and the dual formulation,
\begin{equation*}
 J(F_0^{N_0}, F_1^{N_1, L_1}, F_2^{N_2, L_2})\lesssim  (N_{\min}^{12})^{\frac{3}{2r}}(L_{\min}^{12})^{\frac{1}{r}}(L_{\max}^{12})^{\frac{1}{2r}} \|F_0^{N_0}\|_{L^{r}}\|F_1^{N_1, L_1}\|_{L^p}\|F_2^{N_2, L_2}\|_{L^p}.
\end{equation*}
Hence,
\begin{equation}\label{S_N}
\begin{split}
 S_{\bf N} & \lesssim N_1^{\frac{3}{2r}}\sum_{\bf L} L_1^{\frac{1}{r}-b}L_2^{\frac{1}{2r}-b}\|F_0^{N_0}\|_{L^{r}}\|F_1^{N_1}\|_{L^p}\|F_2^{N_2}\|_{L^p}\\
 & \lesssim N_1^{\frac{3}{2r}}\|F_0^{N_0}\|_{L^{r}}\|F_1^{N_1}\|_{L^p}\|F_2^{N_2}\|_{L^p},
\end{split}
\end{equation}
as $b>\frac{1}{r}$. But then by Lemma \ref{HLH_dyadic},
\begin{align*}
 \sum_{\bf N} \frac{S_{\bf N}}{N_0^{-\sigma} N_1^\sigma N_2^\sigma}
&\lesssim \sum_{\bf N} \chi_{N_1\leq N_0\sim N_2} \frac{N_1^{\frac{3}{2r}-\sigma}}{N_0^0} \|F_0^{N_0}\|_{L^{r}}\|F_1^{N_1}\|_{L^p}\|F_2^{N_2}\|_{L^p}\\
& \lesssim \|F_0\|_{L^{r}}\|F_1\|_{L^p}\|F_2\|_{L^p},
\end{align*}
as $\sigma>\frac{3}{2r}$.

\noindent$\bullet$\quad\textbf{The LHH case: } $N_0\ll N_1\sim N_2$. Using \eqref{S_N}, we have by Lemma \ref{LHH_dyadic}
\begin{align*}
 \sum_{\bf N} \frac{S_{\bf N}}{N_0^{-\sigma} N_1^\sigma N_2^\sigma}
&\lesssim \sum_{\bf N} \chi_{N_0\leq N_1\sim N_2} \frac{N_0^\sigma}{N_1^{2\sigma-\frac{3}{2r}}} \|F_0^{N_0}\|_{L^{r}}\|F_1^{N_1}\|_{L^p}\|F_2^{N_2}\|_{L^p}\\
& \lesssim \|F_0\|_{L^{r}}\|F_1\|_{L^p}\|F_2\|_{L^p},
\end{align*}
since $\sigma >\frac{3}{2r}$.

\end{proof}


\appendix


\section{Local well-posedness for Sobolev data in $H^{\frac{7}{4}+}$}\label{appA}
For simplicity, we will only sketch the proof for the equation
\begin{equation}\label{main_eq_D}
 \Box u=(Du)^2,
\end{equation}
where $D$ stands for the spatial gradient. To treat the general equation \eqref{main_eq}, which may contain time derivatives in the nonlinearity as well, one can apply the arguments below to the equation $\Box (\del_t u)=\del_t(\del u)^2$ with appropriate initial data, for which the estimates can be closed at one derivative less regularity. This, along with the presented argument, will allow one to close the estimates for $\del u$ in $H^{\frac{3}{4}+}$.
 
We start by stating the Strichartz estimates for the wave equation. For proofs of these estimates see \cite{ginibre-velo}, \cite{keel-tao}.

A pair of indices $(p, q)$ is called wave-admissible in dimension $n$, if it belongs to the set
\[
 \sA=\left \{(p, q) : 2\leq p\leq \infty, 2\leq q<\infty, \frac{2}{p}+\frac{n-1}{q}\leq \frac{n-1}{2}\right \}.
\]
We will use the notation $p'$ for the Lebesgue conjugate of the index $p$, that is, $\frac{1}{p}+\frac{1}{p'}=1$.

\begin{theorem}[Strichartz estimates] If $(q, r)$ and $(\~{q}, \~{r})$ are wave-admissible, then
\[
 \|D^s u\|_{L^q_t L^r_x}\lesssim \|(u_0, u_1)\|_{\dot{H}^\gamma\times \dot{H}^{\gamma-1}} + \|D^{\~{s}}\Box u\|_{L^{\~{q}'}_t L^{\~{r}'}_x},
\]
provided the following scaling condition holds,
\[
 \frac{1}{q}+\frac{n}{r}-s=\frac{n}{2}-\gamma=\frac{1}{\~{q}'}+\frac{n}{\~{r}'}-2-\~{s}.
\]
\end{theorem}

The pair $(p, q)=(4, \infty)$ barely fails to be wave-admissible in dimension $n=2$. However, to keep the notation simple, we will assume that it is wave-admissible, and proceed with the estimates using this pair. For the proper proof one can take for example $(p, q)=\left (\frac{12}{3-4\epsilon}, \frac{3}{2\epsilon}\right )$ for a small $\epsilon>0$, which will give $s=\frac{1}{4}+\epsilon$, and the estimates will close in $H^{\frac{7}{4}+\epsilon}$ instead of $H^{\frac{7}{4}}$.

For the following choice of indeces in dimension $n=2$,
\[
 p=4, q=\infty, s=\frac{1}{4}; \qquad \gamma =1; \qquad \~{p}'=1, \~{q}'=2, \~{s}=0, 
\]
for which the scaling condition is satisfied, one gets the estimate
\[
 \|D^{\frac{1}{4}} u\|_{L^4_t L^\infty_x}\lesssim \|(u_0, u_1)\|_{\dot{H}^1\times L^2} + \|\Box u\|_{L^1_t L^2_x}.
\]
Applying this estimate to $D^\frac{3}{4}u$ leads to
\begin{equation}\label{St}
 \|D u\|_{L^4_t L^\infty_x}\lesssim \|(u_0, u_1)\|_{H^\frac{7}{4}\times H^\frac{3}{4}} + \|\Box u\|_{L^1_t H^\frac{3}{4}_x}.
\end{equation}

On the other hand, by the classical energy inequality for \eqref{main_eq_D} in $S_T=[0, T)\times \R^2$,
\[
 \|D u(t, \cdot)\|_{H^\frac{3}{4}}\lesssim \|u(0, \cdot)\|_{H^\frac{7}{4}}+\|\del_t u(0, \cdot)\|_{H^\frac{3}{4}}+\int_0^t \|(Du Du)(t', \cdot)\|_{H^\frac{3}{4}}\, dt'.
\]
But, by H\"{o}lder's inequality, the fractional Leibniz rule, and estimate \eqref{St},
\begin{align*}
 \int_0^t \|(Du Du)(t', \cdot)\|_{H^\frac{3}{4}}\, dt' & \leq\|\Box u\|_{L^1_t H^\frac{3}{4}_x}\\
 & \lesssim T^\frac{3}{4}\|Du\|_{L^4_t L^\infty_x}\|Du\|_{L_t^\infty H_x^\frac{3}{4}}\\
 & \lesssim T^\frac{3}{4} (\|(u_0, u_1)\|_{H^\frac{7}{4}\times H^\frac{3}{4}} + \|\Box u\|_{L^1_t H^\frac{3}{4}_x})\|Du\|_{L_t^\infty H^\frac{3}{4}}.
\end{align*}

Taking $T$ small, the last estimate combined with the previous energy inequality is enough to show the LWP for \eqref{main_eq_D} with data in $H^{\frac{7}{4}}\times H^{\frac{3}{4}}$ using the Pickard's iteration method. For the proper proof for a wave admissible pair $(p, q)$, one additionally needs to use the Sobolev's inequality for the spatial norm in the previous to last estimate in order to place $Du$ in $L_x^q$.


\section{The general well-posedness theorem}\label{appB}
Below we state the general local well-posedness theorem for nonlinear wave equations with data in the Fourier-Lebesgue spaces $\^{H}_s^r\times \^{H}_{s-1}^r$. We only include the main idea of the proof here, as, with minor differences, it closely follows the proof of the analogous theorem for $L^2$ based spaces. For the proof in the $L^2$ case we refer the reader to \cite[Theorem 4.1]{selberg_thesis} (see also \cite[Theorem 5.3]{kl-sel:surrvey}).

Let us consider the Cauchy problem
\begin{align}
 & \Box u=F(u, \del u), \quad (t, x)\in \R^{1+n}, \label{NLW_eq}\\
 & (u, \del_t u)|_{t=0}=(f, g)\in \^{H}_s^r\times \^{H}_{s-1}^r \label{NLW_ic},
\end{align}
where $\del u$ stands for the space-time gradient of $u$, and $F$ is a smooth function with $F(0)=0$.

\begin{theorem}[c.f Theorem 14 in \cite{selberg_thesis}]\label{gen_LWP}
 Assume $s\in\R$, $\frac{1}{r}<b<1$, $\epsilon\in (0, 1-b)$. If
\begin{equation*}\label{F_b}
 \|F(u, \del u)\|_{X_{\sigma-1, b+\epsilon-1}^r}\leq A_\sigma (\|u\|_{Z_{s, b}^r})\|u\|_{Z_{\sigma, b}^r}
\end{equation*}
for all $\sigma\geq s$, and
\begin{equation*}\label{F_c}
 \|F(u, \del u)-F(v, \del v)\|_{X_{\sigma-1, b+\epsilon-1}^r}\leq A_s (\|u\|_{Z_{s, b}^r}+\|v\|_{Z_{s, b}^r})\|u-v\|_{Z_{\sigma, b}^r},
\end{equation*}
for all $u, v\in Z_{s, b}^r$, where $A_\sigma:\R_+\to\R_+$ is increasing and locally Lipschitz for every $\sigma\geq s$, then
\begin{itemize}
 \item (existence) There exists $u\in Z_{s, b}^r$, which solves \eqref{NLW_eq}-\eqref{NLW_ic} on $[0, T]\times \R^n$, where $T=T(\|f\|_{\^{H}_s^r}+\|g\|_{\^{H}_s^r})>0$ depends continuously on $\|f\|_{\^{H}_s^r}+\|g\|_{\^{H}_s^r}$.
 \item (uniqueness) The solution is unique in the class $Z_{s, b}^r$, i.e., if $u, v\in Z_{s, b}^r$ both solve \eqref{NLW_eq}-\eqref{NLW_ic} on $[0, T]\times \R^n$ for some $T>0$, then
\[
 u(t)=v(t) \qquad \text{for } t\in [0, T].
\]
 \item (Lipschitz) The solution map
\[
 (f, g)\mapsto u, \qquad \^{H}_s^r\times \^{H}_{s-1}^r\to Z_{s, b}^r
\]
is locally Lipschitz.
 \item (higher regularity) If the data has higher regularity
\[
 f\in \^{H}_\sigma^r, g\in \^{H}_{\sigma-1}^r \qquad \text{for } \sigma >s,
\]
then $u\in C([0, T], \^{H}_{\sigma}^r)\cap C^1([0, T], \^{H}_{\sigma-1}^r)$ for any $T>0$ for which $u$ solves \eqref{NLW_eq}-\eqref{NLW_ic}. In particular, if $(f, g)\in\sS$, then $u\in C^\infty ([0, T]\times \R^n)$.
\end{itemize}
\end{theorem}

The proof of the general LWP theorem relies on a fixed point argument by showing that the solution  map $\Phi$, given by
\[
 \Phi u(t) = \cos (tD)\cdot f + D^{-1} \sin (tD) g + \int_0^tD^{-1} \sin ((t-t')D)\cdot F(u, \del u)\, dt',
\]
is a contraction of some closed ball in $Z_{s, b; T}^r$ for appropriately chosen $T$. This is achieved through the following estimate for the solution of the inhomogeneous equation $\Box u=F(t, x)$ with initial data $(f, g)$,
\[
  \|\Phi u\|_{Z_{s, b}^r}\leq C\left (\|f\|_{\^{H}_s^r}+\|g\|_{\^{H}_{s-1}^r}+T^{\epsilon/2}\|F\|_{X_{s-1, b+\epsilon-1}^r}\right ).
\]
The above estimate exhibits a time factor in front of the inhomogeneity, if one is allowed to place $F$ in $X_{s-1, b+\epsilon-1}^r$, instead of the natural $X_{s-1, b-1}^r$. Exploiting the time factor in this estimate by choosing $T$ small, and relying on the estimates in the theorem, one can show that $\Phi$ is a contraction of some closed ball in the space $Z_{s, b; T}^r$, and hence has a fixed point.



\end{document}